\documentclass{article}
\usepackage{tikz}
\usetikzlibrary{trees}
\usepackage{amsfonts}
\usepackage{amsmath}
\usepackage{amsthm}
\usepackage{authblk}
\usepackage{graphicx}
\usepackage{multirow}
\newcommand{\R}{\mathbb{R}}
\newtheorem{corollary}{Corollary}
\newtheorem{lemma}{Lemma}
\newtheorem{property}{Property}
\newtheorem{proposition}{Proposition}
\DeclareMathOperator{\sign}{sign}
\newcommand{\ignore}[1]{}
\title{More bounds on the diameters of convex polytopes
}

\author{David Bremner, Antoine Deza, William Hua, and Lars Schewe}
\ignore{
\author[1]{David Bremner}
\affil[1]{University of New Brunswick}

\author[2]{Antoine Deza}
\author[2]{William Hua}
\affil[2]{McMaster University}

\author[3]{Lars Schewe}
\affil[3]{Technische Universit\"at Darmstadt}
}

\begin{document}

\maketitle

\begin{abstract}

Let $\Delta(d,n)$ be the maximum possible edge diameter over
all $d$-dimensional polytopes defined by $n$ inequalities.
The Hirsch conjecture, formulated in 1957, suggests that
$\Delta(d,n)$ is no greater than $n-d$.  No polynomial bound
is currently known for $\Delta(d,n)$, the best one being
quasi-polynomial due to Kalai and Kleitman in 1992.  Goodey
showed in 1972 that $\Delta(4,10)=5$ and $\Delta(5,11)=6$,
and more recently, Bremner and Schewe showed $\Delta(4,11)=
\Delta(6,12)=6$.  In this follow-up, we show that
$\Delta(4,12)=7$ and present strong evidence that $\Delta(5,12)=\Delta(6,13)=7$.

\end{abstract}

Finding a good bound on the maximal edge diameter $\Delta(d,n)$ of a
polytope in terms of its dimension $d$ and the number of its facets
$n$ is one of the basic open questions in polytope theory
\cite{BG}. Although some bounds are known, the behaviour of the
function $\Delta(d,n)$ is largely unknown. The Hirsch conjecture,
formulated in 1957 and reported in \cite{GD}, states that $\Delta(d,n)$
is linear in $n$ and $d$: $\Delta(d,n) \leq n-d$. The conjecture is
known to hold in small dimensions, i.e., for $d \leq 3$ \cite{VK},
along with other specific pairs of $d$ and $n$
(Table~\ref{before}). However, the asymptotic behaviour of
$\Delta(d,n)$ is not well understood: the best upper bound --- due to
Kalai and Kleitman --- is quasi-polynomial \cite{GKDK}.

In this article we will show that $\Delta(4,12)=7$ and present strong
evidence for $\Delta(5,12)=\Delta(6,13)=7$. The first of these new
values is of particular interest since it indicates that the Hirsch bound is
not sharp in dimension $4$.

Our approach is computational and builds on the approach used by
Bremner and Schewe \cite{DBLS}.  Section~\ref{sec:general} introduces
our computational framework and some related background. We then
discuss our results in Section~\ref{sec:results}.

\begin{table}
  \begin{center}
    \begin{tabular}{cc|ccccc}
      & & \multicolumn{5}{|c}{$n-2d$} \\
      & & 0 & 1 & 2 & 3 & 4 \\
      \hline
      \multirow{5}{*}{$d$} & 4 & 4 & 5 & 5 & 6 & 7+ \\
      & 5 & 5 & 6 & 7-8 & 7+ & 8+ \\
      & 6 & 6 & 7-9 & 8+ & 9+ & 9+ \\
      & 7 & 7-10 & 8+ & 9+ & 10+ & 11+ \\
      & 8 & 8+ & 9+ & 10+ & 11+ & 12+
    \end{tabular}
  \end{center}

  \caption{Previously known bounds on $\Delta(d,n)$ \cite{DBLS,
    PG, FHVK, VKDW}. \label{before}}
\end{table}

\section{General approach}
\label{sec:general}

In this section we give a summary of our general approach.
This is substantially similar to that in~\cite{DBLS}, and the reader is 
referred there for more details. 

It is easy to see via a perturbation argument that $\Delta(d,n)$ is
always achieved by some simple polytope.  
By a reduction applied from \cite{VKDW}, we only
need to consider \emph{end-disjoint} facet-paths:  paths where
the end vertices do not lie on a common facet
(\emph{facet-disjointness}).
It will be convenient both
from an expository and a computational view to work in a polar setting
where we consider the lengths of facet-paths on the boundary of
simplicial polytopes. 
We apply the term \emph{end-disjoint} equally to the corresponding facet
paths, where it has the simple interpretation that two end facets do
not intersect.

\newcommand{\rest}{x_1\dots x_{d-1}}
For any set $Z=\{\,x_1\dots x_{r-2},y_1\dots y_4\,\}\subset \R^r$, as a special case
of the Grassmann-Pl\"{u}cker relations~\cite[\S 3.5]{blswz} on determinants we have
\begin{equation}
  \label{eq:gp-det}
  \begin{split}
  \det(\rest,y_1,y_2)\cdot\det(\rest,y_3,y_4)\\
 + \det(\rest,y_1,y_4)\cdot\det(\rest,y_2,y_3)\\
- \det(\rest,y_1,y_3)\cdot\det(\rest,y_2,y_4)=0
  \end{split}
\end{equation}
We are in particular interested in the case where $r=d+1$ and $Z$
represents $(d+3)$-points in $\R^d$ in homogeneous coordinates; the
various determinants are then signed volumes of simplices.  In the
case of points drawn from the vertices of a simplicial polytope, we
may assume without loss of generality that these simplices are never
flat (i.e.\ determinant $0$). Thus if we define $\chi(v_1\dots
v_{d+1})=\sign(\det(v_1\dots v_{d+1}))$ it follows from
\eqref{eq:gp-det} that
\begin{eqnarray*}
  \{  \chi(\rest, y_1, y_2) \chi(\rest, y_3, y_4), & \\
     -\chi(\rest, y_1, y_3) \chi(\rest, y_2, y_4), & \\
      \chi(\rest, y_1, y_4) \chi(\rest, y_2, y_3) \} &
   = \{ -1, +1 \}.
\end{eqnarray*}
Any alternating map $\chi : E^{d+1} \to \{-,+\}$ satisfying these
constraints for all $(d+3)$-subsets is called a \emph{uniform
  chirotope}; this is one of the many axiomatizations of \emph{uniform
  oriented matroids}~\cite{blswz}.  The facets and interior points of
a uniform chirotope are straightforward to define in terms of equality
and non-equality of related signs.  In the rest of this paper we call
uniform chirotopes simply chirotopes.

\begin{figure}[htb]\label{3X}
  \begin{center}
    \includegraphics[width=2in]{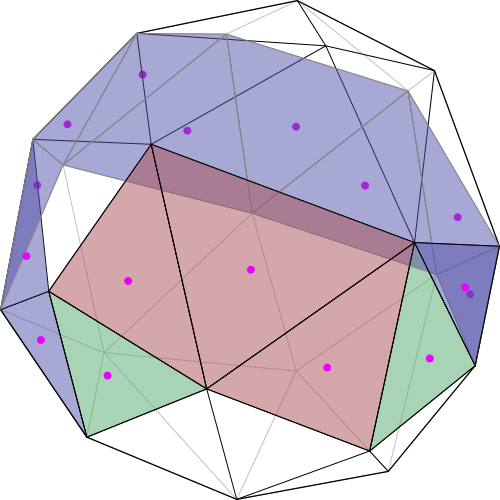}
  \end{center}

  \caption{Illustrating a non-shortest facet-path.}
  \label{fig:shortcut}
\end{figure}

Our general strategy is to show $\Delta(d,n) \neq k$ by generating all
combinatorial types of facet-paths of length $k$ on $n$ vertices in
dimension $d$ and showing that none can be embedded on the boundary of
a chirotope as a shortest path.  Note that if the facet-path uses all
available vertices, then there cannot be any interior points.  In the
general case, we solve a further relaxation of the problem, and show
that even if some points are allowed to be interior, a particular
combinatorial type of path is not embeddable on the boundary of the convex
hull of $n$-points in $\R^d$. In addition to Grassman-Pl\"ucker constraints, 
and those that force the $k$-path onto the boundary, 
we also add constraints preventing the existence of shorter paths
between the starting and ending facets, i.e.\ every potential shortcut is
infeasible by virtue of containing a non-facet.
See Figure~\ref{3X} for an illustration of a shortcut on a
3-dimensional polytope.

Chirotopes can be viewed as a generalization of real polytopes in the
sense that for every real polytope, we can obtain its chirotope
directly.  Therefore, showing the non-existence of chirotopes
satisfying specific properties immediately precludes the existence of
real polytopes holding the same properties.  The search for a
chirotope with a particular facet-path on its boundary is encoded as
an instance of SAT \cite{LS1,LS2}.  The SAT solver used here was
MiniSat \cite{MS}.

The generation of all possible paths for particular $d$ and $n$ begins
with case where the paths are \emph{non-revisiting}, i.e., paths where
no vertex is visited more than once. These can be generated via a
simple recursive scheme, using a bijection with \emph{restricted
  growth strings}.

Multiple revisit paths are generated from paths with one less revisit
by identifying pairs of vertices without introducing extra ridges to
the facet-path or causing the end facets to intersect.

\begin{figure}[htb]\label{39}
  \begin{center}
    \includegraphics[width=2in]{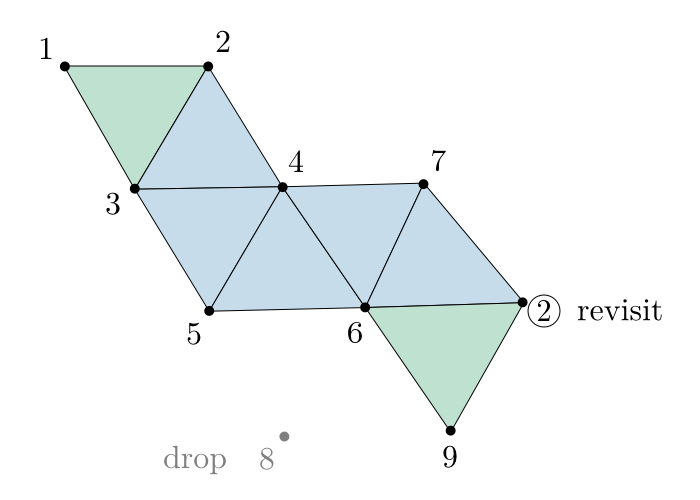}
  \end{center}

  \caption{Example of a facet-path.}
\end{figure}

If a vertex is not used in a facet-path we call this occurrence a
\emph{drop}.
See Figure~\ref{39} for an illustration of a path of length 6 involving
1 revisit (vertex 2) and and 1 drop (vertex 8) with $n=9$ and $d=3$. 
We can then classify paths by dimension $d$, primal-facets/dual-vertices $n$, length $k$,
the number of revisits $m$, and the number of drops $l$.
For end-disjoint paths, a simple counting argument yields:
\begin{eqnarray*}
  m - l &   =  & k +  d - n \\
      m & \leq & k -  d     \\
      l & \leq & n - 2d
\end{eqnarray*}
Table~\ref{paths} provides the number of paths to consider
for each possible combination of values.

\begin{table}[bht]
  \begin{minipage}[b]{0.5\linewidth}
    \centering
    \begin{tabular}{cc|c|cc|c}
      $d$ & $n$ & $k$ & $m$ & $l$ & \# \\
      \hline
      4 & 10 & 6 & 0 & 0 & 15   \\
      4 & 10 & 6 & 1 & 1 & 24   \\
      4 & 10 & 6 & 2 & 2 & 16   \\
      \hline
      4 & 11 & 7 & 0 & 0 & 50   \\
      4 & 11 & 7 & 1 & 1 & 200  \\
      4 & 11 & 7 & 2 & 2 & 354  \\
      4 & 11 & 7 & 3 & 3 & 96   \\
      \hline
      4 & 12 & 8 & 0 & 0 & 160  \\
      4 & 12 & 8 & 1 & 1 & 1258 \\
      4 & 12 & 8 & 2 & 2 & 5172
    \end{tabular}
  \end{minipage}
  \begin{minipage}[b]{0.5\linewidth}
    \centering
    \begin{tabular}{cc|c|cc|c}
      $d$ & $n$ & $k$ & $m$ & $l$ & \# \\
      \hline
      4 & 12 & 8 & 3 & 3 & 7398 \\
      4 & 12 & 8 & 4 & 4 & 1512 \\
      \hline
      5 & 11 & 7 & 1 & 0 & 98   \\
      5 & 11 & 7 & 2 & 1 & 98   \\
      \hline
      5 & 12 & 8 & 1 & 0 & 1079 \\
      5 & 12 & 8 & 2 & 1 & 3184 \\
      5 & 12 & 8 & 3 & 2 & 2904 \\
      \hline
      6 & 12 & 7 & 1 & 0 & 11   \\
      \hline
      6 & 13 & 8 & 1 & 0 & 293  \\
      6 & 13 & 8 & 2 & 1 & 452
    \end{tabular}
  \end{minipage}

  \caption{Number of paths to consider, SAT instances to
    solve. \label{paths}}
\end{table}

With the implementation of~\cite{DBLS}, we were able to reconfirm Goodey's
results for $\Delta(4, 10)$ and $\Delta(5, 11)$ in a matter
of minutes.  While the number of paths to consider increases
with the number of the revisits, in our experiments these paths are 
much less computationally demanding than the ones with 
fewer revisits. For example, the 7,398 paths of length 8
on $4$-polytopes with $12$ facets and involving 3 revisits and 3 drops 
require only a tiny fraction of the computational effort to tackle the
160 paths without a drop or revisit.

In order to deal with the intractability of the problem as the
dimension, number of facets, and path length increased, we proceeded
by splitting our original facet embedding problem into subproblems by
fixing chirotope signs.  We use the (non-SAT based) \texttt{mpc}
backtracking software~\cite{BBG2006} to backtrack to a certain fixed
level of the search tree; every leaf job was then processed in
parallel on the Shared Hierarchical Academic Research Computing
Network (SHARCNET).  Figure~\ref{fig:split} illustrates the splitting
process on a problem generated from the octahedron.  Note that
variable propagation reduces the number of leaves of the tree.

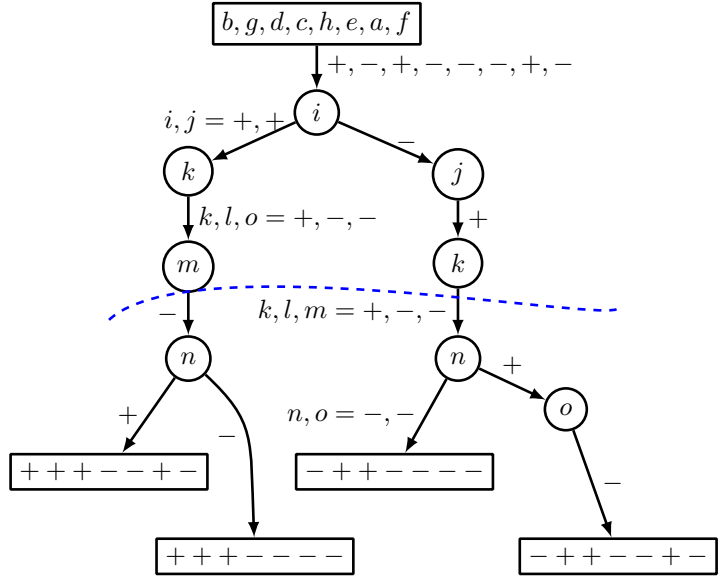
\begin{figure}
  \centering
  \begin{tikzpicture}[anchor=base,>=latex,join=bevel,x=.6,y=0.4]
  \pgfsetlinewidth{1}

  \begin{scope}
    \tikzstyle{every node}=[rectangle,draw]
    \draw (131,604) node (init) {$b,g,d,c,h,e,a,f$};
    % \draw (35,258) node (fail1) {$\emptyset$};
    % \draw (35,345) node (fail2) {$\emptyset$};
    % \draw (255,32) node (fail3) {$\emptyset$};
    \draw (0,180) node[] (suc1) {$+++--+-$};
    \draw (92,100) node[] (suc2) {$+++----$};
    % \draw (126,154) node[rotate=90] (suc3) {$+++--0-$};
    \draw (180,180) node (suc5) {$-++----$};
    % \draw (234,134) node[rotate=90] (suc6) {$-++--0-$};
    \draw (321,100) node (suc4) {$-++--+-$};
  \end{scope}

  \begin{scope}
    \tikzstyle{every node}=[circle,draw,white]
    \draw (220,375) node (k2) {$k$};
    \draw (131,518) node (i) {$i$};
    \draw (220,288) node (n2) {$n$};
    \draw (220,462) node (j) {$j$};
    \draw (50,375) node (m) {$m$}; 
    \draw (288,240) node (o) {$o$};
    \draw (50,462) node (k1) {$k$};
    \draw (50,288) node (n1) {$n$};
\end{scope}

  \begin{scope}
    \tikzstyle{every node}=[circle,draw,anchor=center]
    {\draw (i) node {$i$};}
      \draw (k1) node {$k$}; 
      \draw (m) node {$m$}; 
      {\draw (j) node {$j$};}
      {\draw (k2) node  {$k$};}

      \draw (n2) node {$n$};
      \draw (o) node  {$o$};
      \draw  (n1) node {$n$};
\end{scope}

\begin{scope}
  \tikzstyle{every node}=[right,midway]
  \tikzstyle{every edge}=[draw]
  \draw [->] (init) .. controls (131,578) and (131,563) .. node {$+,-,+,-,-,-,+,-$} (i) ;    

    \draw [->] (i) -- node {$-$} (j);      

    \draw [->] (j) --  node {$+$}(k2);

  \draw [->] (k2) --
        node[left]  {$k,l,m=+,-,-$} (n2);
  
  \draw [->] (n2) -- node[above] {$+$} (o);

  \draw [->] (o) -- node {$-$} (suc4);
  %\draw [->] (o) .. controls (278,121) and (269,88)  .. (fail3) node {$0$};
  \path [->] (n2) edge   node[left] {$n,o=-,-$}  (suc5);

 % \draw [->] (n2) edge  node[sloped,above,midway]  {$n,o=0,-$} (suc6);
    
\end{scope}

\begin{scope}
  \tikzstyle{every node}=[left,midway]
  \tikzstyle{every edge}=[draw]
  % Edge: i -> k1

    \draw [->] (i) -- node[at start] {$i,j=+,+$} (k1);

  % Edge: k1 -> m

    \path [->] (k1) edge  node[right] {$k,l,o=+,-,-$} (m);    
    \draw[blue,dashed](0,330)  .. controls (50,400) and (300,330) .. (320,340);

    % Edge: n1 -> suc1
    \path [->] (m) edge node {$-$} (n1);

  \path [->] (n1) edge  node {$+$} (suc1);
  % Edge: n1 -> suc2
  \draw [->] (n1) .. controls (93,233) and (89,218)  .. node {$-$}(suc2);
  % Edge: n1 -> suc3
  %\draw [->] (n1) .. controls (107,233) and (111,218)  .. node {$0$} (suc3);
  % Edge: m -> fail1
 % \draw [->] (m) .. controls (81,321) and (67,303)  .. node {$0$} (fail1);
  % Edge: k1 -> fail2
  %$\draw [->] (k1) .. controls (80,409) and (66,391)  .. node {$0$} (fail2);
  % Node: fail1

\end{scope}

\end{tikzpicture}
%% Local Variables:
%% mode: LaTeX
%% TeX-master: "paper"
%% End:
  \caption{Using partial backtracking to generate subproblems}
  \label{fig:split}
\end{figure}

Jobs requiring a long time to complete were further split and executed
on the cluster until the entire search space was covered.
Table~\ref{splits} provides the number of paths which were
computationally difficult enough to require splitting. For example,
out of $160$ paths of length 8 on $4$-polytopes with $12$ facets
without drop or revisit, 2 required splitting.

\begin{table}[htb]
  \begin{center}
    \begin{tabular}{cc|c|cc|c}
      $d$ & $n$ & $k$ & $m$ & $l$ & \# \\
      \hline
      4 & 12 & 8 & 0 & 0 & 2   \\
      5 & 12 & 8 & 1 & 0 & 15  \\
      5 & 12 & 8 & 2 & 1 & 6   \\
      6 & 13 & 8 & 1 & 0 & 138 \\
      6 & 13 & 8 & 2 & 1 & 63
    \end{tabular}
  \end{center}

  \caption{Number of ``difficult'' paths. \label{splits}}
\end{table}

\section{Results}
\label{sec:results}
Summarizing the computational results, we have:

\begin{proposition} \label{result}
  There are no $(4,12)$-
  polytopes with facet-disjoint vertices at distance 8.
\end{proposition}

Note that we actually prove something slightly stronger: no chirotope
admits a path of length 8 between vertex-disjoint facets on its boundary
for $d=4, n=12$, i.e., there are no so-called $(4, 12)$-\emph{matroid
polytopes} with vertex-disjoint facets at distance 8.
While the non-existence of $k$-length
paths implies the non-existence of $(k+1)$-length paths, it
is not obvious if the non-existence of end-disjoint
$k$-length paths implies the non-existence of $(k+1)$-length
paths.  To be able to rule out vertices (not necessarily facet-disjoint) at
distance $l > k$, we introduce the following lemma.

\begin{lemma} \label{bound}
  If $\Delta(d-1,n-1) < k$ and there is no $(d,n)$-polytope
  with two facet-disjoint vertices at distance $k$, then
  $\Delta(d,n) < k$.
\end{lemma}

\begin{proof}
  Assume the contrary.  Let $u$ and $v$ be vertices on a
  $(d,n)$-polytope at distance $l \geq k$.  By considering a
  shortest path from $u$ to $v$, there is a vertex $w$ at
  distance $k$ from $u$.  $u$ and $w$ must share a common
  facet $F$ to prevent a contradiction.  $F$ is a
  $(d-1,n-1)$-polytope with diameter at least $k$.
\end{proof}

By Proposition \ref{result} and because $\Delta(3,11) = 6$
\cite{VK}, we can apply Lemma \ref{bound} to obtain the
following new entry for $\Delta(d,n) $.

\begin{corollary} \label{4127}
  $\Delta(4,12) = 7$
\end{corollary}

The computations for $\Delta(5,12)$ and $\Delta(6,13)$
are still underway. In particular, out of the 7,167 $8$-paths
to consider for $5$-polytopes having 12 facets, only 11
paths with 1 revisit and no drop remain to be computed.
If the results for remaining $8$-paths keep on showing
unsatisfiability, it would imply that $\Delta(5,12)\neq 8$
and $\Delta(6,13)\neq 8$.
Since $7 \leq \Delta(5,12) \leq 8$ \cite{DBLS,FHVK}, by
Proposition \ref{result} we could immediately obtain  $\Delta(5,12) = 7$.
We recall the following result of Klee and Walkup~\cite{VKDW}:
\begin{property} \label{0123}
  $\Delta(d,2d+k) \leq \Delta(d-1,2d+k-1) + \lfloor
  k/2 \rfloor + 1$ for $0 \leq k \leq 3$
\end{property}
Applying Property \ref{0123} to $\Delta(5,12) = 7$ would yield a new upper
bound $\Delta(6,13) \leq 8$, from which we could obtain  $\Delta(6,13) = 7$.
Property \ref{0123} along with the 3 new entries for $\Delta(d,n)$ would imply the additional upper bounds:   $\Delta(5,13) \leq 9$, $\Delta(6,14) \leq 11$,
  $\Delta(7,14) \leq 8$, $\Delta(7,15) \leq 12$ and
  $\Delta(8,16) \leq 13$ (see Table~\ref{after}).

\begin{table}[htb]
  \begin{center}
    \begin{tabular}{cc|ccccc}
      & & \multicolumn{5}{|c}{$n-2d$} \\
      & & 0 & 1 & 2 & 3 & 4 \\
      \hline
      \multirow{5}{*}{$d$} & 4 & 4 & 5 & 5 & 6 & \textbf{7} \\
      & 5 & 5 & 6 & \textbf{7} & \textbf{7-9} & 8+ \\
      & 6 & 6 & \textbf{7} & \textbf{8-11} & 9+ & 9+ \\
      & 7 & \textbf{7-8} & \textbf{8-12} & 9+ & 10+ & 11+ \\
      & 8 & \textbf{8-13} & 9+ & 10+ & 11+ & 12+
    \end{tabular}
  \end{center}

  \caption{Summary of bounds on $\Delta(d,n)$ assuming $\Delta(5,12)=\Delta(6,11)=7$.}\label{after}
\end{table}

\newpage

\section{Acknowledgments}

This work was supported by grants from the Natural Sciences
and Engineering Research Council of Canada and MITACS, and by the Canada Research Chair program, 
and made
possible by the facilities of the Shared Hierarchical
Academic Research Computing Network
(\verb|http://www.sharcnet.ca/|).

\noindent {\small 
David Bremner}\\
{\sc Faculty of Computer Science,}
{\sc  University of New Brunswick, Canada}. \\
{\em Email}: bremner{\small @}unb.ca\\

\noindent {\small 
Antoine Deza, William Hua}\\
{\sc Advanced Optimization Laboratory,}
{\sc Department of Computing and Software,}
{\sc  McMaster University, Hamilton, Ontario, Canada}. \\
{\em Email}: deza, huaw{\small @}mcmaster.ca\\

\noindent {\small 
Lars Schewe}\\
{\sc Fachbereich Mathematik,}
{\sc  Technische Universit\"at Darmstadt, Germany}. \\
{\em Email}: schewe{\small @}mathematik.tu-darmstadt.de\\

\end{document}